\documentclass[letterpaper, 10 pt, conference]{ieeeconf}  

\IEEEoverridecommandlockouts                              
\usepackage{amsmath,amssymb}
\usepackage{latexsym}
\usepackage{psfrag}
\usepackage{epsfig}
\usepackage{epstopdf}
\usepackage{graphicx,subfigure}
\usepackage{graphicx}
\usepackage{color}
\usepackage{bm}
\usepackage{cite}

\def\scr#1{{\cal #1}}
\def\bb#1{{\mathbb #1}}

\def\rep#1{(\ref{#1})}
\newtheorem{theorem}{Theorem}

\newtheorem{lemma}{Lemma}
\newtheorem{remark}{Remark}

\newtheorem{assumption}{Assumption}
\newcommand{\matt}[1]{\begin{bmatrix}#1\end{bmatrix}}

\newcommand{\R}{{\rm I\!R}}

\DeclareMathOperator{\diag}{diag}

\begin{document}

\title{On Convergence Rate of a Continuous-Time Distributed Self-Appraisal Model with Time-Varying Relative Interaction Matrices}

\author{Weiguo Xia, Ji Liu,  Tamer Ba\c{s}ar, and Xi-Ming Sun%
\thanks{Weiguo Xia and Xi-Ming Sun are with the School of Control Science and Engineering, Dalian University of Technology, China ({\tt\small \{wgxiaseu,sunxm\}@dlut.edu.cn}). Ji Liu and Tamer Ba\c{s}ar are with the Coordinated Science Laboratory, University of Illinois at Urbana-Champaign, USA ({\tt\small \{jiliu,basar1\}@illinois.edu}).
The work of Liu and Ba\c{s}ar was supported in part by Office of Naval Research (ONR) MURI Grant N00014-16-1-2710, and in part by NSF under grant CCF 11-11342. }}
\maketitle

\begin{abstract}

This paper studies a recently proposed continuous-time distributed self-appraisal model with time-varying interactions among a network of $n$ individuals which are
characterized by a sequence of time-varying relative interaction matrices.
The model describes the evolution of the social-confidence levels of the individuals via a reflected appraisal mechanism in real time.
We first show by example that when the relative interaction matrices are
stochastic (not doubly stochastic),
the social-confidence levels of the individuals may not converge to a steady state.
We then show that when the relative interaction matrices are doubly stochastic,
the $n$ individuals' self-confidence levels will all converge to $1/n$,
which indicates a democratic state,
 exponentially fast
under appropriate assumptions,
and provide an explicit expression of the convergence rate.

\end{abstract}

\section{Introduction}

Opinion dynamics have a long history and have been studied extensively in social sciences.
Probably the most well-known model for opinion dynamics is the classical DeGroot model \cite{De74}.
Various models have been proposed for opinion dynamics to understand how an individual's opinion evolves over time, including the Friedkin-Johnsen model \cite{FrJo99,Fr15}, the Hegselmann-Krause model \cite{HeKr02,BlHeTs09,EtBa15},
the DeGroot-Friedkin model \cite{JiMiFrBu15,JiMiFrBu13,MiJiFrBu14,XuLiBa15}, and the
Altafini model \cite{Al13,MeShJoCaHo16,He14,LiChBaBe15,XiCaJo16}.
However, there is few work concerning the self-confidence levels of the individuals in a social network.

Recently, a new model, called
the DeGroot-Friedkin model, has been proposed in \cite{JiMiFrBu15}.
The model considers the situation when a group of individuals
discusses a sequence of issues, and studies the evolution of the self-confidence levels of individuals
(i.e., how confident an individual is for her opinions on the sequence of issues)
via the reflected appraisal mechanism proposed in \cite{Fr11}.
Lately, a modified DeGroot-Friedkin model has been proposed in \cite{XuLiBa15} which provides a
time-efficient, distributed implementation of the original DeGroot-Friedkin model. The model has been studied in \cite{XuLiBa15} with a fixed doubly stochastic relative interaction matrix,
and in \cite{XiLiJoBa16} with time-varying doubly stochastic relative interaction matrices, while
the analysis of the modified DeGroot-Friedkin model with
a fixed stochastic (not doubly stochastic) relative interaction matrix still remains open.

Both the DeGroot-Friedkin model and the modified DeGroot-Friedkin model
are described in discrete times.
Sometimes a continuous-time model would be a natural choice
especially when the opinions of individuals evolve gradually over time;
see for example \cite{Lo07,BlHeTs10}.
Recently, a continuous-time distributed self-appraisal model has been proposed in
\cite{ChLiBeXuBa17}, which shows that when the relative interaction matrix is fixed
and there is no ``dominant neighbor'' in the network,
the social-confidence levels of the individuals will asymptotically converge to a steady state,
depending on the relative interaction matrix, under appropriate connectivity assumption.
Local exponential stability of the steady state was shown by checking the Jacobian matrix.
But no convergence rate result was obtained.
Analysis of the continuous-time distributed self-appraisal model for the case of
general stochastic relative interaction matrix (i.e., without the assumption of no ``dominant neighbor'') also remains open.

In a realistic social network, the interaction among the individuals may change from time to time.
With this in mind, this paper aims to study the continuous-time distributed self-appraisal model with time-varying interactions which are described by a sequence of time-dependent relative interaction matrices,
and specifically derive a convergence rate of the model.
We first construct a simple example to show that the  model may not converge for general time-varying stochastic (not doubly stochastic) relative interaction matrices.
With this observation in mind, we focus our attention on the case when the relative interaction matrices
are doubly stochastic, and show that the self-appraisals of the individuals all converge to $\frac{1}{n}$
exponentially fast, where $n$ is the number of individuals in a network,
and obtain an explicit expression of the convergence rate.
Although doubly stochastic relative interaction matrices may be artificial,
this case has an important social meaning as it explains how a democratic state is formed
in a social network \cite{JiMiFrBu15}.

The main contribution of this paper is to provide an explicit expression of the convergence rate
of the continuous-time distributed self-appraisal model \cite{ChLiBeXuBa17} with time-varying doubly stochastic
relative interaction matrices.
We extend the result in \cite{ChLiBeXuBa17} in three-fold.
First, our result implies that the model converges for all fixed doubly stochastic
relative interaction matrices, whereas the result in \cite{ChLiBeXuBa17} does not subsume
this implication because not all doubly stochastic matrices satisfy the no ``dominant neighbor'' assumption.
Second, we show that the convergence is exponentially fast for doubly stochastic relative interaction matrices,
whereas only asymptotic convergence was proved in \cite{ChLiBeXuBa17}.
Lastly, we show that exponential convergence holds for time-varying
doubly stochastic relative interaction matrices and obtain an explicit expression of the convergence rate,
which was not considered in \cite{ChLiBeXuBa17}.

The remainder of this paper is organized as follows. Some notations are
introduced in Section \ref{notation}. In Section \ref{model}, the continuous-time self-appraisal model is introduced.
In Section \ref{motivating}, an  example is presented to motivate the assumptions. The main results of the paper are presented in Section \ref{result}, whose analysis and proofs
are given in Section \ref{analysis}. Some discussions are given in Section \ref{discussion}. The paper ends with some concluding remarks in Section \ref{conclusion}.

\subsection{Notations} \label{notation}

For a positive integer $n$, let $\mathcal V$  denote the set $\{1,\ldots,n\}$.
We use $\Delta_n$ to denote the simplex $\{x\in\R^n:\ x_i\geq0, i\in\mathcal V, \sum_{i=1}^nx_i=1\}$.
For each $i\in\scr{V}$, we use $e_i$ to denote the vector in $\R^n$ whose $i$th element equals 1 and all the other elements equal 0. Let $I$ denote the identity matrix and let $\bm{1}$  denote the all-one vector with appropriate dimensions. A row-stochastic matrix is a nonnegative matrix with each row sum equal 1, and is simply called a stochastic matrix. A matrix is column-stochastic if its transpose is a row-stochastic matrix. A matrix is called doubly stochastic if  it is both row-stochastic and column-stochastic. For any two real vectors $x,y\in\R^n$, we write $x\geq y$ if $x_i\geq y_i$ for all $i\in\scr{V}$ and $x>y$ if $x_i> y_i$ for all $i\in\scr{V}$. We use $\diag(x)$ to denote the diagonal matrix with the $i$th entry being $x_i$. For a scalar $a\in\R$, let ${\lfloor a\rfloor}$ denote the largest integer that is no larger than $a$.

\section{The Model} \label{model}

In this section, we introduce the continuous-time distributed self-appraisal model proposed in \cite{ChLiBeXuBa17}.

Consider a network consisting of $n>1$ individuals with the constraint that each individual
can communicate only with certain other individuals called ``neighbors''.
The neighbor relationships among the $n$ individuals are described by a time-dependent,
$n$-vertex, directed graph $\bb G(t)$ whose
vertices correspond to individuals and whose arcs depict neighbor relationships.
Specifically, we say that individual $j$ is an {\em outgoing neighbor} of individual $i$ at time $t$ if
there is an arc from vertex $i$ to vertex $j$ in $\bb G(t)$, and say that
individual $k$ is an {\em incoming neighbor} of individual $i$ at time $t$ if
there is an arc from vertex $k$ to vertex $i$ in $\bb G(t)$.
We use $\scr N_i^{{\rm in}}(t)$ and $\scr N_i^{{\rm out}}(t)$ to denote the sets of
incoming and outgoing neighbors of individual $i$ at time $t$, respectively.
Each individual $i$ has control over a real-valued quantity $x_i(t)$
which represents the self-appraisal of individual $i$.
The self-appraisal $x_i(t)$ takes values in the interval $[0,1]$, which
measures how confident individual $i$ is on her opinions.
The larger $x_i(t)$ is, the more confident is individual $i$.
The continuous-time distributed self-appraisal model is as follows:
\begin{equation}\label{eq:sys1}
\dot{x}_i(t)=-(1-x_i(t))x_i(t)+\sum_{j\in\scr N_i^{{\rm in}}(t)} c_{ji}(t)(1-x_j(t))x_j(t),
\end{equation}
where $c_{ji}(t)$ is the {\em relative inter-personal weight} \cite{JiMiFrBu15} that individual $j$ assigns to her outgoing neighbor\footnote{
Note that $j$ is an incoming neighbor of $i$ in \rep{eq:sys1}, and thus $i$ is an outgoing neighbor of $j$.
}
$i$ at time $t$  which is a positive real number.

The relative inter-personal weights satisfy the following condition:
\begin{equation}\label{eq:C2}
\sum_{j\in\scr N_i^{{\rm out}}(t)} c_{ij}(t)=1,\;\;\; i\in\scr V.
\end{equation}
Note that each $c_{ij}(t)$ in \rep{eq:C2} is in the interval $(0,1]$,
and can be set by individual $i$ herself.
Let $c_{ij}(t)=0$ for all pairs of $i$ and $j$ such that $j\notin \scr N_i^{{\rm out}}(t)$.
Then, condition \rep{eq:C2} implies that
$\sum_{j=1}^n c_{ij}(t) =1$ for all $i\in\scr V$ and time $t$, and thus
each matrix $C(t)=\matt{c_{ij}(t)}_{n\times n}$ is a stochastic matrix
whose diagonal entries all equal zero.
The matrix $C(t)$ is called the {\em relative interaction matrix} \cite{JiMiFrBu15} at time $t$.

At initial time $t=0$,
the self-appraisals are scaled so that they sum to one, i.e., $\sum_{i\in\scr V}x_i(0)=1$.
It will be shown that this initial condition guarantees that
$\sum_{i\in\scr V}x_i(t)=1$ for all time $t>0$.

\begin{remark}
System (\ref{eq:sys1}) with a fixed relative interaction (i.e., $c_{ji}(t)\equiv c_{ji}$ for all time $t$)
was proposed and studied in \cite{ChLiBeXuBa17}. The system can be viewed as a continuous-time version of the modified DeGroot-Friedkin model studied in \cite{XuLiBa15,XiLiJoBa16}.
\hfill $\Box$
\end{remark}

To help readers to grasp the social meaning of the model \rep{eq:sys1}
and understand the motivations,
we give a brief interpretation of the model below. See \cite{ChLiBeXuBa17} for detailed explanation.

We begin with the following continuous-time opinion dynamics:
\begin{equation}\label{eq:sys}
\dot{z}_i(t)=-(1-x_i(t))\Big(-z_i(t)+\sum_{j\in\scr N_i^{{\rm out}}(t)} c_{ij}(t)z_j(t)\Big),\ i\in \scr V,
\end{equation}
where $z_i(t)$ is a real number representing the opinion of individual $i$ on an issue of interest at time $t$.
Note that system \rep{eq:sys} is a continuous-time consensus process \cite{ReBe05}
with the dynamics of $z_i(t)$ scaled by the nonnegative factor $(1-x_i(t))$.
Thus, $(1-x_i(t))$ can be viewed as a measure of the
total amount of opinions individual $i$ accepts from others at time $t$,
and $c_{ij}(t)(1-x_i(t))$ can be regarded as the corresponding portion individual $i$ accepts
from neighbor $j$, which is consistent with the social meaning of $x_i(t)$, i.e.,
$x_i(t)$ is the self-appraisal of individual $i$ measuring how confident she is on her current opinion.

We now turn to the justification of the model \rep{eq:sys1}.
From \rep{eq:sys1}, the dynamics of $x_i(t)$ is determined by two terms: $(1-x_i(t))x_i(t)$ and $\sum_{j\in\scr N_i^{{\rm in}}(t)}c_{ji}(t)(1-x_j(t))x_j(t)$.
We consider the latter first.
Recall that $c_{ji}(t)(1-x_j(t))$ measures the amount of opinion individual $j$ accepts from neighbor $i$
in the opinion dynamics \rep{eq:sys} and $x_j(t)$ is the self-appraisal of individual $j$
reflecting the importance of individual $j$ in the network.
Product $c_{ji}(t)(1-x_j(t))x_j(t)$ can then be viewed as the measure of
importance of individual $i$ to individual $j$,
and thus $\sum_{j\in\scr N_i^{{\rm in}}(t)}c_{ji}(t)(1-x_j(t))x_j(t)$
can be viewed as the measure of importance of individual $i$ to the others in the network.
We then consider the other term $(1-x_i(t))x_i(t)$.
In view of condition (\ref{eq:C2}), it follows that
$$(1-x_i(t))x_i(t)=\sum_{j\in\scr N_i^{{\rm out}}(t)}c_{ij}(t)(1-x_i(t))x_i(t).$$
Thus, $(1-x_i(t))x_i(t)$ can be interpreted as the measure of importance of others to individual $i$.

From the preceding discussion, the model (\ref{eq:sys1}) is designed
for each individual to calculate, in a distributed manner, the difference
between her level of importance to others and others' level of importance to her.
Note that any equilibrium state of system (\ref{eq:sys1}) is a state when the
difference equals zero for each individual.
Therefore, the distributed self-appraisal model (\ref{eq:sys1})
aims to drive all individuals' differences to zero.

To proceed,
let $x(t)=[x_1(t),\ldots,x_n(t)]^\top$ and $X(t)=\diag(x(t))$.
Then, system (\ref{eq:sys1}) can be written in the form of an
$n$-dimensional state equation:
\begin{align}\label{eq:sys2}
\dot{x}(t)=&-(I-X(t))x(t)+C(t)^\top(I-X(t))x(t),\nonumber\\
=&-W(t,x(t))x(t),
\end{align}
where $W(t,x(t))\triangleq I-X(t)-C(t)^\top(I-X(t))$. Throughout the paper, we assume that $C(t)$ is piecewise constant, i.e., there exists an infinite time sequence $t_0,t_1,t_2,\dots,$ with $t_0=0$ such that
\begin{align}\label{eq:C}
C(t)=C(t_k), \quad t\in[t_k,t_{k+1}),\ k\geq0.
\end{align}
Then, system (\ref{eq:sys1}) can be rewritten as
{\small
\begin{align}\label{eq:sys3}
\dot{x}_i(t)=&-(1-x_i(t))x_i(t)+\sum_{j\in\scr N_i^{{\rm in}}(t)}c_{ji}(t_k)(1-x_j(t))x_j(t),
\end{align}
}
or in a compact form
\begin{align}\label{eq:sys4}
\dot{x}(t)=&-(I-X(t))x(t)+C(t_k)^\top(I-X(t))x(t),
\end{align}
for $t\in[t_k,t_{k+1}).$ Let $\tau_k\triangleq t_{k+1}-t_k$. $\tau_k$ is a positive number called a dwell time.

\begin{remark}
Since the matrix $C(t)$ is stochastic, it can be verified that $\bm{1}^\top W(t,x(t))=\bm{1}^\top[I-X(t)-C(t)^\top(I-X(t))]=0.$ The fact that $\Delta_n$ is positive invariant as will be proved in Lemma \ref{lm:5} later implies $1-x_i(t)\geq0, i\in\mathcal V$ and thus it follows that $W^\top(t,x(t))$ is a Laplacian matrix \cite{GoRo01} for any $t\geq0$. It is worth noting that $W(t,x(t))$ is not necessarily a Laplacian matrix even if $C(t)$ is doubly stochastic. The difference between system (\ref{eq:sys2}) and the continuous-time consensus algorithm $\dot{x}(t)=-L(t)x(t)$ is that $W^\top(t,x(t))$ is a state-dependent Laplacian matrix, thus resulting in a nonlinear system, while the Laplacian matrix $L(t)$ is
not state-dependent. The derived convergence results for the consensus system are typically based on assumptions on the elements of the Laplacian matrix such as $-l_{ij}(t)\in[\underline{\alpha},\bar{\alpha}]\cup\{0\},$ for $t\geq0$, where $\underline{\alpha}$ and $\bar{\alpha}$ are positive constants \cite{ReBe05,OlMu04,XiWa08}. While system (\ref{eq:sys2}) involves a state-dependent matrix $W(t,x(t))$ and so does (\ref{eq:sys4}).  Whether the condition that the boundedness of the nonzero off-diagonal elements of $-W(t,x(t))$ from below for all time $t\geq0$ is satisfied or not is unknown and hard to check. Thus, those existing results of continuous-time consensus processes \cite{ReBe05,OlMu04,XiWa08} cannot be applied here.
Although there are some convergence results for opinion dynamics models with state-dependent connectivity and for consensus systems with cut-balanced properties available in the literature \cite{BlHeTs09,HeTs13,MaGi13}, we do not see a way to apply these results and their analysis to system (\ref{eq:sys2}).
In this paper, we will resort to an analysis technique (see Section \ref{analysis}) to bound the extreme values of $x_i(t)$ so that the convergence rate can be characterized. The technique is partially inspired by the work of  \cite{ShJo13d} as system (\ref{eq:sys3}) is transformed to a form of equations (\ref{eq:lm6-1}) that also appear in the analysis of consensus systems.
\hfill $\Box$
\end{remark}

In this paper, we will look into the dynamic behavior of system (\ref{eq:sys3}) and analyze how the self-appraisals of individuals evolve with time-varying relative interaction matrices. We will focus our attention on the case when $C(t)$ is doubly stochastic, because of the motivating example in the next section, and establish an exponential convergence result for the state of system (\ref{eq:sys3}).

\section{A Motivating Example} \label{motivating}

In this section, we provide an  example to motivate the assumption that $C(t)$ is a doubly stochastic matrix for all $t$ proposed in the next section.

When $C(t)\equiv C$ is fixed for all $t\geq t_0$, it has been shown in \cite{ChLiBeXuBa17} that system (\ref{eq:sys1}) converges to an equilibrium other than $e_i,\ i=1,\dots,n$, for almost all initial conditions in $\Delta_n$ under the constraints that every agent has at least two incoming neighbors and the inter-individual weights  $c_{ij}$ are upper bounded by $\frac{1}{2}$. However, for a time-varying relative interaction matrix $C(t)$, the convergence of the system cannot be guaranteed in general, which is illustrated by the following example.

Let
$$ C_1=\begin{bmatrix}
        0 &\frac{3}{4}& 0& \frac{1}{4} \\
        \frac{1}{4}& 0 &\frac{3}{4}& 0 \\
        0 &\frac{1}{4} &0 &\frac{3}{4}\\
        \frac{3}{4} &0& \frac{1}{4} &0 \\
      \end{bmatrix},\ C_2=\begin{bmatrix}
                            0& 1& 0& 0  \\
                            \frac{1}{2}& 0 &\frac{1}{2} &0 \\
                            0&\frac{1}{3}& 0 &\frac{2}{3} \\
                            0 &0 &1 &0\\
                          \end{bmatrix},$$
and
\begin{equation}\label{eq:C}
C(t)=\begin{cases}
C_1, &t\in[2k*0.4,(2k+1)*0.4),\\
C_2, &t\in[(2k+1)*0.4,(2k+2)*0.4),
\end{cases}
\end{equation}
for all integers $k\geq0$. Then, $t_k-t_{k-1}=\tau=0.4,\ k\geq0$. Since both $C_1$ and $C_2$ are irreducible, at each time instant $t\geq0$ the graph is strongly connected. Note that $C_1$ is a doubly stochastic matrix, one knows that system (\ref{eq:sys1}) with a fixed $C(t)\equiv C_1$ will converge to $\frac{1}{4}\bm{1}$ \cite{ChLiBeXuBa17}. $C_2$ is not doubly stochastic and for system (\ref{eq:sys1}) with a fixed $C(t)\equiv C_2$, the state will converge to $[0.0917,0.211,0.486,0.211]^\top$.
For a random initial condition in $\Delta_4\backslash\{e_1,\ldots,e_n\}$, when $C(t)$ takes the form of (\ref{eq:C}), the system state does not converge as shown in Fig.~\ref{fig:1}.

\begin{figure}[!htb]
  \centering
  \includegraphics[width=6.5cm]{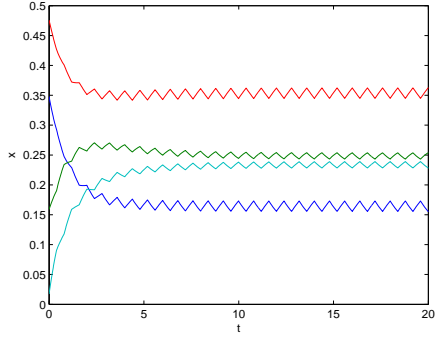}
  \caption{The system state with a time-varying $C(t)$ switching between a doubly stochastic matrix $C_1$ and a non-doubly stochastic matrix $C_2$.}
  \label{fig:1}
\end{figure}

The reason can be explained as follows. The equilibria other than $e_i,\ i=1,\dots,4$, of system (\ref{eq:sys1}) corresponding to $C_1$ and $C_2$ are different. Then, one can imagine that when $C(t)$ switches between $C_1$ and $C_2$, the system state will oscillate, since  in each switching mode, the state tends to converge to the corresponding equilibrium.

Note that when $C(t)\equiv C,\ t\geq t_0$, the system state will converge to $\frac{1}{n}\bm{1}$ as long as $C$ is irreducible and doubly stochastic. This motivates us to focus on the case of doubly stochastic matrices $C(t)$ and impose some connectivity conditions to guarantee the convergence of system (\ref{eq:sys3}).

\section{Main Results} \label{result}

In this section, we present the main result of the paper.  The following assumptions will be considered in the following discussion on system (\ref{eq:sys3}).

\begin{assumption}\label{ass:1}
Each $C(t_k)$, $k\geq0$, is a doubly stochastic matrix with zero diagonal elements, and there exists a constant $\gamma>0$ such that $c_{ij}(t_k)\geq\gamma$ for all nonzero $c_{ij}(t_k)$.
\end{assumption}

\begin{assumption}\label{ass:2}
There exists an integer $B\geq1$ such that the union graph $\bigcup_{k=l}^{l+B-1} \mathbb G(t_k)$ is strongly connected for all nonnegative integers $l\geq0$.
\end{assumption}

\begin{assumption}\label{ass:3}
There exist two positive constants $\bar{\tau}_D$ and $\underline{\tau}_D$ such that $\bar{\tau}_D\geq \tau_k\geq\underline{\tau}_D$ for all $k\geq0$.
\end{assumption}

Let
$$h(t)=\max_{i\in\scr V}\{x_i(t)\},\ l(t)=\min_{i\in\scr V}\{x_i(t)\},\ V(t)=h(t)-l(t).$$ The function $V(t)$ is a measure of the maximum difference between the self-appraisals of the individuals in the network. If $V(t)\rightarrow0$ as $t$ goes to infinity, then the self-appraisals of the individuals all converge to a common value that  is $\frac{1}{n}$ as will be shown. The main result of the paper is stated as follows.

\begin{theorem}\label{thm:1}
Suppose that $n\geq3$ and Assumptions \ref{ass:1}, \ref{ass:2} and \ref{ass:3} hold. Then,
\begin{itemize}
\item[(a)] $\Delta_n$ is a positive invariant set of system (\ref{eq:sys3}), i.e., for any initial condition $x(t_0)\in\Delta_n$, $x(t)\in\Delta_n$ for all $t\geq t_0$.
    \item[(b)]If $x(t_0)\in\Delta_n\backslash\{e_1,\dots,e_n\}$ and $x(t_0)$ has $m$ nonzero entries, $m\geq2$, then $\lim_{t\rightarrow \infty}x_i(t)=\frac{1}{n}$ for all $i\in\scr{V},$ and the convergence is exponentially fast with a rate given by
\begin{equation}\label{eq:thm1}
V(t)\leq\Big(1-\alpha \mu^{n-1}\Big)^{-\frac{1+2B(n-1)}{B(n-1)}}e^{-\lambda t}V(t_0),
\end{equation}
for all $t\geq t_0$, where
\begin{align}\label{eq:alpha}
&\alpha\triangleq e^{-\bar{\tau}_DB(n-1)(1-2l(t_{(n-m)B}))},\nonumber\\
 &\mu\triangleq\alpha\gamma (1-e^{-\underline{\tau}_Dl(t_{(n-m)B})}),\nonumber\\
& \lambda\triangleq\frac{\ln (1-\alpha \mu^{n-1})^{-1}}{B\bar{\tau}_D(n-1)}
 \end{align} with $l(t_{(n-m)B})=\min_{i\in\mathcal V}\{x_i(t_{(n-m)B})\}>0$.
\end{itemize}
\end{theorem}

\begin{remark}
The intuition for the convergence of  system  (\ref{eq:sys3})  to a democratic state $[\frac{1}{n},\frac{1}{n},\ldots,\frac{1}{n}]^\top$ is that the relative interaction matrix $C(t_k)$ has a common left eigenvector $\bm{1}$ for all $k\geq0$ under Assumption \ref{ass:1}. However it is not straightforward to derive the conclusion established in Theorem \ref{thm:1}. For  the case when $C(t)\equiv C$ is fixed for all $t\geq t_0$ and is stochastic, but not necessarily doubly stochastic, the analysis of system (\ref{eq:sys1}) is still open. Note that not all doubly stochastic matrices satisfy the no ``dominant neighbor'' assumption in \cite{ChLiBeXuBa17}, and hence the result in \cite{ChLiBeXuBa17} does not subsume the conclusion that system  (\ref{eq:sys3})  converges for all fixed doubly stochastic relative interaction matrices. In addition, we provide an explicit expression of the exponential convergence rate of system  (\ref{eq:sys3}),
whereas only asymptotic convergence was proved in \cite{ChLiBeXuBa17}.
\hfill $\Box$
\end{remark}

\section{Analysis} \label{analysis}

We begin with some preliminaries. The upper Dini derivative of a continuous function $V(t,x(t)):\R\times \R^m\rightarrow \R$  with respect to $t$ is defined as
$$D^+V(t,x(t)) =\limsup_{s\rightarrow0^+}\frac{V(t+s,x(t+s))-V(t,x(t))}{s}.$$ The next result is useful for the calculation of Dini derivatives of a function \cite{Da66,LiFrMa07}.

\begin{lemma}\label{lm:3}
Let $V_i(t, x):\R\times\R^m\rightarrow \R,\ i\in\mathcal V$ be of class $C^1$ and $V(t,x)=\max_{i\in\mathcal V} V_i(t, x)$. If $I(t)=\{i\in\mathcal V|V(t,x(t))=V_i(t, x(t))\}$ is the set of indices where the maximum is reached at $t$, then $D^+V(t,x(t))=\max_{i\in I(t)}\{\dot{V}_i(t,x_i(t))\}$.
\end{lemma}

The next lemma proven in \cite{XiLiJoBa16} will be very useful in the following discussion.

\begin{lemma}\label{lm:4}
Suppose that $\beta\in\R^n$, $\beta\geq0$, $\sum_{k=1}^n\beta_k=1$, and $x\in\R^n$, $x\geq0$, $\sum_{k=1}^nx_k=1$. Then, there exists a constant
$$v\in  \left[\min\limits_{k\in\scr{V} \atop \beta_k\neq0}\{x_k\},\max\limits_{k\in\scr{V} \atop \beta_k\neq0}\{x_k\}\right]$$
such that $v\leq\sum_{k=1}^n\beta_kx_k$ and
\begin{equation}\label{lm:4-1}
v-v^2=\sum_{k=1}^n\beta_k(x_k-x_k^2).
\end{equation}
\end{lemma}

We identify below the equilibria of system (\ref{eq:sys3}).

\begin{lemma}\label{lm:2}
$e_i$ is an equilibrium of system (\ref{eq:sys3}) for each $i\in \scr V$. In addition, if Assumption \ref{ass:1} holds, then $\frac{1}{n}\bm{1}$ is an equilibrium of system (\ref{eq:sys3}).
\end{lemma}
\begin{proof}
Note that $(I-\diag(e_i))e_i=e_i-e_i=0$. It follows that
\begin{align*}
-(I-\diag(e_i))e_i+C(t)^\top(I-\diag(e_i))e_i=0.
\end{align*}
Therefore, $e_i$ is an equilibrium of system (\ref{eq:sys3}) for each $i\in V$. If Assumption \ref{ass:1} holds, then
\begin{align*}
&-(I-\diag(\frac{1}{n}\bm{1}))\frac{1}{n}\bm{1}+C(t_k)^\top(I-\diag(\frac{1}{n}\bm{1}))\frac{1}{n}\bm{1}\\
=&-\frac{n-1}{n^2}\bm{1}+C(t_k)^\top\frac{n-1}{n^2}\bm{1}\\
=&0,
\end{align*}
where the last equality makes use of the assumption that $C(t_k)^\top\bm{1}=\bm{1}$ for all $t_k\geq t_0$.
\end{proof}

Before showing that the set $\Delta_n$ is positive invariant we prove a basic property of system (\ref{eq:sys1}).

\begin{lemma}\label{lm:1}
For system (\ref{eq:sys1}), $\bm{1}^\top x(t)=\bm{1}^\top x(t_0)$ for all $t\geq t_0$.
\end{lemma}
\begin{proof}
Direct calculation gives that
\begin{align*}
&\bm{1}^\top\dot{x}(t)\\
=&-\bm{1}^\top(I-X(t))x(t)+\bm{1}^\top C(t)^\top(I-X(t))x(t)\\
=&-\bm{1}^\top(I-X(t))x(t)+\bm{1}^\top(I-X(t))x(t)\\
=&0,
\end{align*}
for all $t\geq t_0$, where the second equality makes use of  the assumption that $C(t)$ is a stochastic matrix for all $t\geq t_0$. The desired conclusion follows.
\end{proof}

We are now in a position to prove item (a) in Theorem \ref{thm:1} and some important properties of the functions $l(t)$ and $h(t)$.

\begin{lemma}\label{lm:5}
Suppose that Assumption \ref{ass:1} holds. $\Delta_n$ is a positive invariant set of system (\ref{eq:sys3}).  In addition, for the initial condition $x(t_0)\in\Delta_n$, $l(t)$ is a nondecreasing function and $h(t)$ is a nonincreasing function.
\end{lemma}

\begin{proof} Let $I_1(t)=\{i\in\mathcal V|x_i(t)=h(t)\}$ and $I_2(t)=\{i\in\mathcal V|x_i(t)=l(t)\}$.  By Lemma \ref{lm:1}, we have that  $\bm{1}^\top x(t)=1$ for $t\geq t_0$ if $x(t_0)\in\Delta_n$.  Let $t^\ast\geq0$ be the time instant such that for $t\in[t_0,t^\ast)$, $x(t)\in\Delta_n$, and $h(t^\ast)=1$ or $l(t^\ast)=0$. First consider the case when $h(t^\ast)=1$. Then, one knows that there is only one element, say $i,$ lies in $I_1(t^\ast)$, and hence $x_i(t^\ast)=1$ and $x_j(t^\ast)=0$ for $j\in\scr V\backslash I_1(t^\ast).$ By Lemma \ref{lm:3},
\begin{equation}\label{eq:lm5-4}
D^+h(t^\ast)=\dot{x}_i(t^\ast)=0.
\end{equation}

Next assume that $l(t^\ast)=0$. Since Assumption \ref{ass:1} holds, $\sum_{j=1}^nc_{ji}(t)=1$ for all $i$ and $t\geq t_0$. The vector $x(t^\ast)$ satisfies that $x(t^\ast)\geq0$ and $\sum_{j=1}^nx_j(t^\ast)=1$. For each $i\in \scr V$, it follows from Lemma \ref{lm:4} that there exists a constant
$$v_i(t^\ast)\in \left[\min\limits_{j\in\mathcal N_{i}^{{\rm in}}(t^\ast)}\{x_j(t^\ast)\},\max\limits_{j\in\mathcal N_{i}^{{\rm in}}(t^\ast)}\{x_j(t^\ast)\}\right]$$
such that $v_i(t^\ast)\leq\sum_{j\in\scr N_i^{{\rm in}}(t)}c_{ji}(t^\ast)x_j(t^\ast)$, and
\begin{equation}\label{eq:lm5-1}
\sum_{j\in\scr N_i^{{\rm in}}(t)}c_{ji}(t^\ast)(1-x_j(t^\ast))x_j(t^\ast)=v_i(t^\ast)-v_i^2(t^\ast).
\end{equation}
Then, for each $i\in I_2(t^\ast)$, $\dot{x}_i(t)$ at $t=t^\ast$ is given by
\begin{align*}
&-(1-x_i(t^\ast))x_i(t^\ast)+\sum_{j\in\scr N_i^{{\rm in}}(t)}c_{ji}(t^\ast)(1-x_j(t^\ast))x_j(t^\ast)\\
=&-(1-x_i(t^\ast))x_i(t^\ast)+v_i(t^\ast)-v_i^2(t^\ast)\\
=&-(x_i(t^\ast)-v_i(t^\ast))(1-x_i(t^\ast)-v_i(t^\ast)).
\end{align*}
Since
$$0=x_i(t^\ast)\leq \min\limits_{j\in\mathcal N_{i}^{{\rm in}}(t^\ast)}\{x_j(t^\ast)\}\leq v_i(t^\ast)$$
and
$$x_i(t^\ast)+v_i(t^\ast)\leq x_i(t^\ast)+\max\limits_{j\in\mathcal N_{i}^{{\rm in}}(t^\ast)}\{x_j(t^\ast)\}\leq 1,$$ it follows that $\dot{x}_i(t^\ast)\geq 0$. In view of Lemma \ref{lm:4},
\begin{equation}\label{eq:lm5-5}
D^+(-l(t^\ast))=\max_{i\in I_2(t^\ast)}\{-\dot{x}_i(t^\ast)\}\leq0.
\end{equation}
Then, (\ref{eq:lm5-4}) and (\ref{eq:lm5-5})  imply that for all $t\geq t_0$, $0\leq l(t)\leq h(t)\leq 1$. Combining with Lemma \ref{lm:1}, $\Delta_n$ is a positive invariant set.

We next show that $h(t)$ is a nondecreasing function. For each $i\in I_1(t)$ and $t\in[t_k,t_{k+1})$,
\begin{align}\label{eq:lm5-2}
\dot{x}_i(t)=&-(1-x_i(t))x_i(t)+\sum_{j\in\scr N_i^{{\rm in}}(t)}c_{ji}(t)(1-x_j(t))x_j(t)\nonumber\\
=&-(x_i(t)-v_i(t))(1-x_i(t)-v_i(t)),
\end{align}
where
$$v_i(t)\in \left[\min\limits_{j\in\mathcal N_{i}^{{\rm in}}(t)}\{x_j(t)\},\max\limits_{j\in\mathcal N_{i}^{{\rm in}}(t)}\{x_j(t)\}\right]$$
satisfies that $v_i(t)\leq\sum_{j\in\scr N_i^{{\rm in}}(t)}c_{ji}(t)x_j(t)$, and
\begin{equation}\label{eq:5-3}
\sum_{j\in\scr N_i^{{\rm in}}(t)}c_{ji}(t)(1-x_j(t))x_j(t)=v_i(t)-v_i^2(t).
\end{equation}
Since
$$x_i(t)\geq \max\limits_{j\in\mathcal N_{i}^{{\rm in}}(t)}\{x_j(t)\}\geq v_i(t)$$
and
$$x_i(t)+v_i(t)\leq x_i(t)+\max\limits_{j\in\mathcal N_{i}^{{\rm in}}(t)}\{x_j(t)\}\leq 1,$$ it follows that $\dot{x}_i(t)\leq 0$. It follows from Lemma \ref{lm:4} that $D^+h(t)=\max_{i\in I_1(t)}\dot{x}_i(t)\leq0,$ and hence $h(t)$ is nonincreasing. The conclusion that $l(t)$ is a nondecreasing function can be proved in a similar way.
\end{proof}

The previous lemma has shown that $\Delta_n$ is positive invariant. The next result says that as long as the initial state $x(t_0)$ is not a vertex of $\Delta_n$, the system state will enter into the interior of the simplex $\Delta_n$ in finite time.

\begin{lemma}\label{lm:7}
Assume that Assumptions \ref{ass:1}, \ref{ass:2}, and \ref{ass:3} hold. Suppose that $x({t_0})\in\Delta_n\backslash\{e_1,\dots,e_n\}$ and has $m$ nonzero entries. Then, $x(t)>0$, for $t\geq t_{(n-m)B}$.
\end{lemma}

\begin{proof}
We first prove the conclusion that if $x_i(t^\ast)>0$ for some $t^\ast\geq t_0$, then $x_i(t)>0$ for $t\geq t^\ast$.
Note that $x(t)\in\Delta_n$ for all $t\geq t_0$ by Lemma \ref{lm:5}, and $C(t_k)$ are stochastic matrices for $k\geq0$. We have
\begin{align*}
\dot{x}_i(t)=&-(1-x_i(t))x_i(t)+\sum_{j\in\scr N_i^{{\rm in}}(t)}c_{ji}(t)(1-x_j(t))x_j(t)\\
\geq&-x_i(t),
\end{align*}
implying that $x_i(t)\geq e^{-t}x_i(t^\ast)$, $t\geq t^\ast$. Therefore, $x_i(t)>0$ for $t\geq t^\ast$ if $x_i(t^\ast)>0$. Define $\mathcal S(t)=\{i\in\scr V|x_i(t)>0\}$. Then, $|\mathcal S(t)|$ is a nondecreasing function and from the assumption of the lemma, $|\mathcal S(t_0)|=m$. It suffices to show that $|\mathcal S(t_{(n-m)B})|=n$.

Define
$$t_{k_1}=\min\{t_k| c_{ji}(t_k)> 0, \text{for some } j\in\mathcal S(t_0),\ i\in \scr V\backslash \mathcal S(t_0)\}.$$
From Assumption \ref{ass:2}, $t_{k_1}$ is well defined, and ${k_1}\leq {B-1}$. At $t_{k_1}$, suppose $c_{ji}(t_{k_1})>0$ for some  $j\in\mathcal S(t_0)$ and $i\in \scr V\backslash \mathcal S(t_0)$. Note that for $i\in\scr V\backslash\mathcal S(t_0)$ and $t\in[t_0,t_{k_1}]$, $x_i(t)=0$. The derivative of $x_i(t)$ at $t=t_{k_1}$ is given by
{\small
\begin{align*}
&\dot{x}_i(t)|_{t=t_{k_1}}\\
=&\Big(-(1-x_i(t))x_i(t)+\sum_{k\in\scr N_i^{{\rm in}}(t)}c_{ki}(t)(1-x_k(t))x_k(t)\Big)\Big|_{t=t_{k_1}}\\
=&\sum_{k\in\scr N_i^{{\rm in}}(t)}c_{ki}(t_{k_1})(1-x_k(t_{k_1}))x_k(t_{k_1})\\
\geq&c_{ji}(t_{k_1})(1-x_j(t_{k_1}))x_j(t_{k_1})\\
>&0.
\end{align*}}
Therefore, there exists a positive constant $\delta$ such that $x_i(t)>0$ for  $t\in (t_{k_1},t_{k_1}+\delta)$. It immediately follows from the previous discussion that $x_i(t)>0$ for $t >t_{k_1}$. Hence, $|\mathcal S(t_{k_1+1})|>|\mathcal S(t_{k_0})|$.

Repeating this process, we can find an  sequence of integers ${k_1},\dots,{k_s}$ such that  $$n=|\mathcal S(t_{k_s+1})|>|\mathcal S(t_{k_{s-1}+1})|>\cdots>|\mathcal S(t_{k_1+1})|>|\mathcal S(t_{k_0})|.$$ Note that $|\mathcal S(t_{k_0})|=m$. Thus, $s\leq n-m$. From Assumption \ref{ass:2}, ${k_q},\ q=1,\dots,s,$ can be chosen such  that $k_{q+1}\leq {k_q+B},\ q=1,\dots,s-1,$. One concludes that ${k_{s}}\leq (n-m)B-1$ and
$t_{k_{s}+1}\leq t_{{(n-m)B}}$. The desired conclusion follows.
\end{proof}

In what follows, we will look at the evolution of $h(t)$ and provide an explicit upper bound for $h(t)$, and thus  the decrease of $V(t)$ over some time interval can be quantitatively characterized.

\begin{lemma}\label{lm:6}
Suppose $n\geq3$. Assume that Assumptions \ref{ass:1}, \ref{ass:2}, and \ref{ass:3} hold. If $x(t_{k_0})>0$ and $x(t_{k_0})\in\Delta_n$ for some integer $k_0\geq0$, then the following inequality holds:
\begin{equation}\label{eq:lm6}
V(t_{k_0+(n-1)B})\leq\Big(1-\alpha \mu^{n-1}\Big)V(t_{k_0}),
\end{equation}
where  $\alpha$ and $\mu$ are given in (\ref{eq:alpha}) with $l(t_{(n-m)B})$ replaced by $l(t_{k_0})$.
\end{lemma}
\begin{proof}
In view of Lemma \ref{lm:5}, $h(t)$ is a nonincreasing function and $l(t)$ is a nondecreasing function. We will bound $h(t_{k_0+(n-1)B})$ from above so that the inequality (\ref{eq:lm6}) can be established. We divide the analysis into three steps.

{\it Step 1.} Let $\mathcal V_0=\{i\in\mathcal V| x_i(t_{k_0})=l(t_{k_0})\}$. For any $i_0\in\mathcal V_0$ and $t\in[t_{k_0},t_{{k_0}+(n-1)B}],$ it follows from Lemma \ref{lm:4} that there exists
\begin{equation}\label{eq:lm6-2}
v_{i_0}(t)\in\left[\min\limits_{j\in\mathcal N_{i_0}^{{\rm in}}(t)}\{x_j(t)\},\max\limits_{j\in\mathcal N_{i_0}^{{\rm in}}(t)}\{x_j(t)\}\right]
  \end{equation}
such that  $v_{i_0}(t)\leq\sum_{j\in\mathcal N_{i_0}^{{\rm in}}(t)}c_{ji_0}(t)x_j(t)$, and
$$\sum_{j\in\scr N_{i_0}^{{\rm in}}(t)}c_{ji_0}(t)(1-x_j(t))x_j(t)=v_{i_0}(t)-v_{i_0}^2(t).$$
Then,
\begin{align}\label{eq:lm6-1}
\dot{x}_{i_0}(t)=&-x_{i_0}(t)+x_{i_0}^2(t)+\sum_{j\in\scr N_{i_0}^{{\rm in}}(t)}c_{ji_0}(t)(1-x_j(t))x_j(t)\nonumber\\
=&-(1-x_{i_0}(t))x_{i_0}(t)+v_{i_0}(t)-v_{i_0}^2(t)\nonumber\\
=&-(x_{i_0}(t)-v_{i_0}(t))(1-x_{i_0}(t)-v_{i_0}(t)).
\end{align}

In view of (\ref{eq:lm6-2}) and Lemma \ref{lm:5}, one has that
$$l(t_{k_0})\leq l(t)\leq v_{i_0}(t)\leq h(t)\leq h(t_{k_0})$$ for $t\geq t_{k_0}.$ In addition,
\begin{align}\label{eq:lm6-3}
2l(t_{k_0})\leq &x_{i_0}(t)+v_{i_0}(t)\nonumber\\
\leq& x_{i_0}(t)+\max\limits_{j\in\mathcal N_{i_0}^{{\rm in}}(t)}\{x_j(t)\}\leq1.
\end{align}
One can then bound $\dot{x}_{i_0}(t)$ from above  as
\begin{align*}
\dot{x}_{i_0}(t)=&-(x_{i_0}(t)-v_{i_0}(t))(1-x_{i_0}(t)-v_{i_0}(t))\\
\leq&-(x_{i_0}(t)-h(t_{k_0}))(1-x_{i_0}(t)-v_{i_0}(t)).
\end{align*}
It follows from Gr\"{o}nwall inequality that
\begin{align*}
x_{i_0}(t)\leq& e^{-\int_{t_{k_0}}^t(1-x_{i_0}(s)-v_{i_0}(s))ds}x_{i_0}(t_{k_0})\\
&\qquad+\Big(1-e^{-\int_{t_{k_0}}^t(1-x_{i_0}(s)-v_{i_0}(s))ds}\Big)h(t_{k_0})\\
=&e^{-\int_{t_{k_0}}^t(1-x_{i_0}(s)-v_{i_0}(s))ds}l(t_{k_0})\\
&\qquad+\Big(1-e^{-\int_{t_{k_0}}^t(1-x_{i_0}(s)-v_{i_0}(s))ds}\Big)h(t_{k_0})
\end{align*}
for $t\in[t_{k_0},t_{{k_0}+(n-1)B}].$ Inequality (\ref{eq:lm6-3}) implies that
$$1- x_{i_0}(t)-v_{i_0}(t)\leq 1-2l(t_{k_0})$$
for $t\in[t_{k_0},t_{{k_0}+(n-1)B}]$ and therefore one has
\begin{align*}
e^{-\int_{t_{k_0}}^t(1-x_{i_0}(s)-v_{i_0}(s))ds}\geq& e^{-(1-2l(t_{k_0}))(t-t_{k_0})}\\
\geq & e^{-(1-2l(t_{k_0}))(t_{k_0+(n-1)B}-t_{k_0})}\\
\geq & e^{-\bar{\tau}_DB(n-1)(1-2l(t_{k_0}))}\\
=&\alpha,
\end{align*}
where the last inequality makes use of Assumption \ref{ass:3} that $\underline{\tau}_D\leq\tau_k\leq\bar{\tau}_D$ for all $k\geq0$. We obtain the following bound for $x_{i_0}(t)$:
\begin{align}\label{eq:lm6-4}
x_{i_0}(t)\leq& e^{-(1-2l(t_{k_0}))(t-t_{k_0})}l(t_{k_0})\nonumber\\
&\qquad\qquad+\Big(1-e^{-(1-2l(t_{k_0}))(t-t_{k_0})}\Big)h(t_{k_0})\nonumber\\
\leq& \alpha l(t_{k_0})+(1-\alpha)h(t_{k_0})
\end{align}
for $t\in[t_{k_0},t_{{k_0}+(n-1)B}]$.

{\it Step 2. }Define
\begin{align*}
&k_1=\min\{k\geq {k_0}|c_{ji}(t_k)> 0, \text{for some } j\in\mathcal V_0,\ i\in \mathcal V\backslash \mathcal V_0\},\\
&\mathcal V_1=\{ i\in \mathcal V\backslash \mathcal V_0|c_{ji}(t_{k_1})>0, \text{for some } j\in\mathcal V_0\}.
\end{align*}
From Assumption \ref{ass:2}, ${k_1}$ is well defined and satisfies that ${k_0}\leq {k_1}\leq{k_0+B-1}$.

For any $i_1\in\mathcal V_1$, it follows from Lemma \ref{lm:4} that
\begin{align}\label{eq:lm6-5}
\dot{x}_{i_1}(t)=&-x_{i_1}(t)+x_{i_1}^2(t)+\sum_{j\in\scr N_{i_1}^{{\rm in}}(t)}c_{ji_1}(t)(1-x_j(t))x_j(t)\nonumber\\
=&-(x_{i_1}(t)-v_{i_1}(t))(1-x_{i_1}(t)-v_{i_1}(t)),
\end{align}
where
\begin{equation}\label{eq:lm6-6}
v_{i_1}(t)\in\left[\min\limits_{j\in\mathcal N_{i_1}^{{\rm in}}(t)}\{x_j(t)\},\max\limits_{j\in\mathcal N_{i_1}^{{\rm in}}(t)}\{x_j(t)\}\right]
  \end{equation}
 satisfies that
\begin{equation}\label{eq:lm6-7}
v_{i_1}(t)\leq\sum_{j\in\mathcal N_{i_1}^{{\rm in}}(t)}c_{ji_1}(t)x_j(t),
\end{equation}
 and
$$\sum_{j\in\scr N_{i_1}^{{\rm in}}(t)}c_{ji_1}(t)(1-x_j(t))x_j(t)=v_{i_1}(t)-v_{i_1}^2(t).$$
Similarly, one has that
$l(t_{k_0})\leq v_{i_1}(t)\leq h(t_{k_0})$ and $ x_{i_1}(t)+v_{i_1}(t)\geq2l(t_{k_0})$, for $t\geq t_{k_1}$.
Note that $n\geq3$ and $\sum_{i\in\scr V}x_i(t)=1,\ t\geq t_0$. One has that
\begin{align}\label{eq:lm6-8}
x_{i_1}(t)+v_{i_1}(t)
\leq& x_{i_1}(t)+\max\limits_{j\in\mathcal N_{i_1}^{{\rm in}}(t)}\{x_j(t)\}\nonumber\\
\leq& 1-l(t_{k_1})\leq1-l(t_{k_0}).
\end{align}

From the definition of $\mathcal V_1$, there exists some $i_0\in\mathcal V_0$ such that $c_{i_0i_1}(t_{k_1})>0$.
It then follows from (\ref{eq:lm6-7}) that
\begin{align*}
v_{i_1}(t)&\leq\sum_{k\in\mathcal N_{i_1}^{{\rm in}}(t)}c_{ki_1}(t_{k_1})x_k(t)\\
&\leq c_{i_0i_1}(t_{k_1})x_{i_0}(t)+(1- c_{i_0i_1}(t_{k_1}))h(t_{k_1})\\
&\leq \gamma\Big(\alpha l(t_{k_0})+(1-\alpha )h(t_{k_0})\Big)+(1-\gamma) h(t_{k_0})\\
&=\alpha\gamma l(t_{k_0})+(1-\alpha \gamma) h(t_{k_0}),
\end{align*}
for $t\in[t_{k_1},t_{k_1+1}],$ where the third inequality makes use of Assumption \ref{ass:1} that $c_{i_0i_1}(t_k)\geq\gamma,\ k\geq0$, inequality (\ref{eq:lm6-4}) and the fact that $h(t_{k_1})\leq h(t_{k_0})$. Combining with (\ref{eq:lm6-5}), one has that
\begin{align}\label{eq:lm6-9}
\dot{x}_{i_1}(t)
\leq&-\Big(x_{i_1}(t)-(\alpha\gamma l(t_{k_0})+(1-\alpha \gamma) h(t_{k_0}))\Big)\nonumber\\
&\qquad\qquad\qquad\cdot\Big(1-x_{i_1}(t)-v_{i_1}(t)\Big),
\end{align}
 for $t\in[t_{k_1},t_{k_1+1}].$ This implies that for $t\in[t_{k_1},t_{k_1+1}],$
\begin{align}\label{eq:lm6-10}
x_{i_1}(t)\leq& e^{-\int_{t_{k_1}}^t(1-x_{i_1}(s)-v_{i_1}(s))ds}x_{i_1}(t_{k_1})+\Big(\alpha\gamma l(t_{k_0})\nonumber\\
&+(1-\alpha \gamma) h(t_{k_0})\Big)\Big(1-e^{-\int_{t_{k_1}}^t(1-x_{i_1}(s)-v_{i_1}(s))ds}\Big)\nonumber\\
\leq&e^{-\int_{t_{k_1}}^t(1-x_{i_1}(s)-v_{i_1}(s))ds}h(t_{k_0})+\Big(\alpha\gamma l(t_{k_0})\nonumber\\
&+(1-\alpha \gamma) h(t_{k_0})\Big)\Big(1-e^{-\int_{t_{k_1}}^t(1-x_{i_1}(s)-v_{i_1}(s))ds}\Big).
\end{align}
In view of  inequality (\ref{eq:lm6-8}), we have
\begin{align*}
e^{-\int_{t_{k_1}}^{t_{k_1+1}}(1-x_{i_1}(s)-v_{i_1}(s))ds}
\leq& e^{-l(t_{k_0})(t_{k_1+1}-t_{k_1})}\\
\leq& e^{-\underline{\tau}_Dl(t_{k_0})}.
\end{align*}
We can then obtain an upper bound for $x_{i_1}(t_{k_1+1})$ as
\begin{align*}\label{eq:lm6-10}
x_{i_1}(t_{k_1+1})\leq&e^{-\underline{\tau}_Dl(t_{k_0})}h(t_{k_0})+\Big(\alpha\gamma l(t_{k_0})\nonumber\\
&\qquad\quad(1-\alpha \gamma) h(t_{k_0})\Big)\Big(1-e^{-\underline{\tau}_Dl(t_{k_0})}\Big)\nonumber\\
\leq &\mu  l(t_{k_0})+(1-\mu ) h(t_{k_0}).
\end{align*}

Then, for  $t\in[t_{k_1+1},t_{k_0+(n-1)B}],$ similar to the analysis in step 1, one can obtain that
\begin{align*}
x_{i_1}(t)\leq& e^{-\int_{t_{k_1+1}}^t(1-x_{i_1}(s)-v_{i_1}(s))ds}x_{i_1}(t_{k_1+1})\\
&\qquad+\Big(1-e^{-\int_{t_{k_1+1}}^t(1-x_{i_1}(s)-v_{i_1}(s))ds}\Big)h(t_{k_1+1})\\
\leq&e^{-(t_{k_0+(n-1)B}-t_{k_1+1})(1-2l(t_{k_0}))}x_{i_1}(t_{k_1+1})\\
&\qquad+\Big(1-e^{-(t_{k_0+(n-1)B}-t_{k_1+1})(1-2l(t_{k_0}))}\Big)h(t_{k_0})\\
\leq&e^{-\bar{\tau}_DB(n-1)(1-2l(t_{k_0}))}\Big(\mu  l(t_{k_0})+(1-\mu ) h(t_{k_0})\Big)\\
&\qquad\qquad+\Big(1-e^{-\bar{\tau}_DB(n-1)(1-2l(t_{k_0}))}\Big)h(t_{k_0})\\
=&\alpha\mu  l(t_{k_0})+(1-\alpha\mu ) h(t_{k_0}).
\end{align*}

{\it Step 3.} Continuing the analysis on time interval $[t_{k_2},t_{k_0+(n-1)B}]$, where $k_2$ is defined as
\begin{align*}
k_2&=\min\{k\geq k_1+1|c_{ji}(t_k)> 0, \text{for some } \\
& \qquad\qquad\qquad\qquad j\in\mathcal V_0\cup\mathcal V_1,\ i\in \mathcal V\backslash (\mathcal V_0\cup\mathcal V_1)\},
\end{align*}
we can similarly define
\begin{align*}
\mathcal V_2&=\{ i\in \mathcal V\backslash (\mathcal V_0\cup\mathcal V_1)|c_{ji}(t_{k_2})>0, \text{for some } j\in\mathcal V_0\cup\mathcal V_1\}.
\end{align*}
Then, using similar arguments to the analysis in step 2,  one can establish an upper bound for $x_{i_2}(t)$, $t\in[t_{k_2+1},t_{k_0+(n-1)B}]$ as
\begin{equation}\label{eq:lm:6-11}
x_{i_2}(t)\leq\alpha\mu^2l(t_{k_0})+(1-\alpha\mu^2 ) h(t_{k_0}).
\end{equation}
Continuing this process, a time sequence $t_{k_0},t_{k_1},\dots,t_{k_p},$  and a sequence of sets $\mathcal V_0,\dots,\mathcal V_p$ are defined as
\begin{align*}
k_{s+1}&=\min\{k\geq k_s+1|c_{ji}(t_{k})> 0, \text{for some } \\
& \qquad\qquad\qquad\qquad j\in\cup_{l=0}^s\mathcal V_l,\ i\in \mathcal V\backslash \cup_{l=0}^s\mathcal V_l\},\\
\mathcal V_{s+1}&=\{ i\in \mathcal V\backslash \cup_{l=0}^s\mathcal V_l|c_{ji}(t_{k_{s+1}})>0, \text{for some }\\
&\qquad\qquad\qquad\qquad j\in\cup_{l=0}^s\mathcal V_l\},
\end{align*}
for $0\leq s\leq p-1$, such that $\mathcal V=\cup_{i=0}^p \mathcal V_i$. By Assumption \ref{ass:2}, $k_s$ satisfies ${k_s+1}\leq {k_{s+1}}\leq{k_s+B}$, for $s=1,\dots,p-1$. Note that $p\leq n-1$ and hence
$${k_p}+1\leq {k_{p-1}}+B+1\leq {k_0}+pB\leq {k_0}+(n-1)B.$$
For all $i\in\mathcal V$ and any $t\in[t_{k_p+1},t_{k_0+(n-1)B}]$, we have the following inequality
\begin{equation}\label{eq:lm6-12}
x_{i}(t)\leq\alpha\mu^p l(t_{k_0})+(1-\alpha\mu^p ) h(t_{k_0}).
\end{equation}
It follows that
\begin{align}\label{eq:lm6-13}
h(t_{k_0+(n-1)B})\leq &h(t_{k_0+pB})\nonumber\\
\leq& \alpha\mu^p l(t_{k_0})+(1-\alpha\mu^p ) h(t_{k_0})\nonumber\\
\leq&\alpha\mu^{n-1} l(t_{k_0})+(1-\alpha\mu^{n-1} ) h(t_{k_0}).
\end{align}
One can then provide a bound for $V(t_{k_0+(n-1)B})$:
\begin{align}\label{eq:lm6-13}
&V(t_{k_0+(n-1)B})\nonumber\\
=&h(t_{k_0+(n-1)B})-l(t_{k_0+(n-1)B})\nonumber\\
\leq &\alpha\mu^{n-1} l(t_{k_0})+(1-\alpha\mu^{n-1} ) h(t_{k_0})-l(t_{k_0})\nonumber\\
\leq&(1-\alpha\mu^{n-1} ) V(t_{k_0}).
\end{align}
This completes the proof.
\end{proof}

We are now in a position to prove Theorem \ref{thm:1}.

{\it Proof of Theorem \ref{thm:1}:}
(a) It has been proved in Lemma \ref{lm:5}.

(b) In view of Lemma \ref{lm:7}, $x(t_{(n-m)B})>0$ and hence $l(t_{(n-m)B})>0$. For $t\geq t_0$, let $s$ be the integer such that $t_s\leq t<t_{s+1}$. Then, from Assumption \ref{ass:3}, $t_{s+1}\leq \bar{\tau}_D(s+1)$, implying that $s\geq \frac{t}{\bar{\tau}_D}-1$.

For $t\geq t_{(n-m)B}$, in view of Lemma \ref{lm:6}, one has
\begin{align*}
V(t)\leq&\Big(1-\alpha \mu^{n-1}\Big)^{\left\lfloor \frac{s-(n-m)B}{(n-1)B}\right\rfloor} V(t_{(n-m)B}).
\end{align*}
Since $s\geq \frac{t}{\bar{\tau}_D}-1$, we have that
\begin{align*}
\left\lfloor \frac{s-(n-m)B}{(n-1)B}\right\rfloor\geq&\left\lfloor\frac{\frac{t}{\bar{\tau}_D}-1-(n-m)B}{(n-1)B}\right\rfloor\\
\geq&\frac{t}{B\bar{\tau}_D(n-1)}-\frac{1+2B(n-1)}{B(n-1)}.
\end{align*}
It follows that
\begin{align*}
V(t)\leq &\Big(1-\alpha \mu^{n-1}\Big)^{\frac{t}{B\bar{\tau}_D(n-1)}-\frac{1+2B(n-1)}{B(n-1)}}V(t_{(n-m)B})\\
\leq &\Big(1-\alpha \mu^{n-1}\Big)^{-\frac{1+2B(n-1)}{B(n-1)}}e^{-\lambda t}V(t_0).
\end{align*}

For $t\in[t_0,t_{(n-m)B})$, inequality (\ref{eq:thm1}) holds since $V(t)\leq V(t_0)$ and $$(1-\alpha \mu^{n-1})^{ -\frac{1+2B(n-1)}{B(n-1)}}e^{-\lambda t}\geq1.$$
This completes the proof.
\hfill $\blacksquare$

\section{Discussions} \label{discussion}

The exponential convergence result of Theorem \ref{thm:1} are obtained based on Assumptions \ref{ass:1}, \ref{ass:2}, and \ref{ass:3}. The assumption that the relative interaction matrix $C(t),\ t\geq t_0,$ is doubly stochastic is critical. If it does not hold, the switched system (\ref{eq:sys3}) may not converge as we have seen in Section \ref{motivating}. It is worth noting that for the case when $C(t)\equiv C$ is fixed for all $t\geq t_0$ and is stochastic, but not necessarily doubly stochastic, the analysis of system (\ref{eq:sys1}) is still not complete. Some convergence result of the system has  been established in \cite{ChLiBeXuBa17}  under some constraint on the relative interaction matrix $C$.

With the example in Section \ref{motivating} in mind, for a general time-varying relative interaction matrix $C(t)$,  additional conditions need to be imposed to guarantee the convergence of system (\ref{eq:sys3}). Assumption \ref{ass:1} is such a condition. Whether the convergence of system (\ref{eq:sys3}) can be established under more relaxed conditions remains unknown. Note that $\frac{1}{n}\bm{1}$ is a left eigenvector of the eigenvalue one of every doubly stochastic matrix $C(t_k),\ k\geq0$. This motivates us to conjecture that if $C(t_k),\ k\geq0,$ have a common left eigenvector corresponding to the eigenvalue one, then  the state of system (\ref{eq:sys3}) converges under  Assumption \ref{ass:2} and Assumption \ref{ass:3}. A numerical example is given to validate this conjecture.

Let $$ C_1=\begin{bmatrix}
        0 &1 & 0& 0 \\
        0& 0 &\frac{1}{2}& \frac{1}{2}\\
        \frac{1}{2} &\frac{1}{2} &0 &0\\
        0 &0& 1 &0 \\
      \end{bmatrix},\ C_2=\begin{bmatrix}
                            0& 1& 0& 0  \\
                            \frac{1}{2}& 0 &\frac{1}{2} &0 \\
                            0&\frac{1}{2}& 0 &\frac{1}{2} \\
                            0 &0 &1 &0\\
                          \end{bmatrix}.$$
$C_1$ and $C_2$ have a common left eigenvector $[\frac{1}{6},\frac{1}{3},\frac{1}{3},\frac{1}{6}]^\top$ corresponding to one. Assume that $C(t)$ is the same as in (\ref{eq:C}). For a random initial condition in $\Delta_4\backslash\{e_1,\dots,e_4\}$, the state evolution is shown in Fig.~\ref{fig:2}. It can be seen that the system state converges to an equilibrium point in $\Delta_4\backslash\{e_1,\dots,e_4\}$.


\begin{figure}[!htb]
  \centering
  \includegraphics[width=6.5cm]{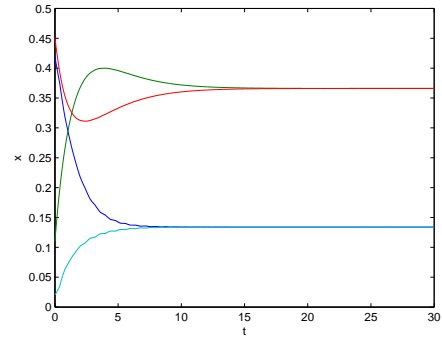}
  \caption{The system state with a time-varying $C(t)$ switching between $C_1$ and $C_2$ with a common left eigenvector to one.}
  \label{fig:2}
\end{figure}

\section{Conclusion} \label{conclusion}

In this paper,  the continuous-time self-appraisal model proposed in \cite{ChLiBeXuBa17} with a time-varying relative interaction matrix has been studied. It has been shown that the self-appraisals of the $n$ individuals in a network will all reach $\frac{1}{n}$ exponentially fast if the time-varying relative interaction matrix is piece-wise constant and doubly stochastic. Similar convergence result has been conjectured for the case when all different relative interaction matrices are row-stochastic and share the same dominant left eigenvector. An explicit expression of the convergence rate has been established. We are interested in further looking into the self-appraisal model with a general relative interaction matrix which is row-stochastic, not necessarily doubly stochastic.


\bibliographystyle{unsrt}

\begin{thebibliography}{10}

\bibitem{De74}
M.~H. DeGroot.
\newblock Reaching a consensus.
\newblock {\em Journal of the American Statistical Association},
  69(345):118--121, 1974.

\bibitem{FrJo99}
N.~E. Friedkin and E.~C. Johnsen.
\newblock Social influence networks and opinion change.
\newblock {\em Advances in Group Processes}, 16(1):1--29, 1999.

\bibitem{Fr15}
N.~E. Friedkin.
\newblock The problem of social control and coordination of complex systems in
  sociology: {A} look at the community cleavage problem.
\newblock {\em IEEE Control Systems Magazine}, 35(3):40--51, 2015.

\bibitem{HeKr02}
R.~Hegselmann and U.~Krause.
\newblock Opinion dynamics and bounded confidence: {M}odels, analysis and
  simulation.
\newblock {\em Journal of Artificial Societies and Social Simulation}, 5:1--24,
  2002.

\bibitem{BlHeTs09}
V.~D. Blondel, J.~M. Hendrickx, and J.~N. Tsitsiklis.
\newblock On {K}rause's multi-agent consensus model with state-dependent
  connectivity.
\newblock {\em IEEE Transactions on Automatic Control}, 54(11):2586--2597,
  2009.

\bibitem{EtBa15}
S.~R. Etesami and T.~Ba\c{s}ar.
\newblock Game-theoretic analysis of the {H}egselmann-{K}rause model for
  opinion dynamics in finite dimensions.
\newblock {\em IEEE Transactions on Automatic Control}, 60(7):1886--1897, 2015.

\bibitem{JiMiFrBu15}
P.~Jia, A.~MirTabatabaei, N.~E. Friedkin, and F.~Bullo.
\newblock Opinion dynamics and the evolution of social power in influence
  networks.
\newblock {\em SIAM Review}, 57(3):367--397, 2015.

\bibitem{JiMiFrBu13}
P.~Jia, A.~MirTabatabaei, N.~E. Friedkin, and F.~Bullo.
\newblock On the dynamics of influence networks via reflected appraisal.
\newblock In {\em Proc. of the 2013 American Control Conference}, pages
  1251--1256, 2013.

\bibitem{MiJiFrBu14}
A.~MirTabatabaei, P.~Jia, N.~E. Friedkin, and F.~Bullo.
\newblock On the reflected appraisals dynamics of influence networks with
  stubborn agents.
\newblock In {\em Proc. of the 2014 American Control Conference}, pages
  3978--3983, 2014.

\bibitem{XuLiBa15}
Z.~Xu, J.~Liu, and T.~Ba\c{s}ar.
\newblock On a modified {D}e{G}root-{F}riedkin model of opinion dynamics.
\newblock In {\em Proc. of the 2015 American Control Conference}, pages
  1047--1052, 2015.

\bibitem{Al13}
C.~Altafini.
\newblock Consensus problems on networks with antagonistic interactions.
\newblock {\em IEEE Transactions on Automatic Control}, 58(4):935--946, 2013.

\bibitem{MeShJoCaHo16}
Z.~Meng, G.~Shi, K.~H. Johansson, M.~Cao, and Y.~Hong.
\newblock Behaviors of networks with antagonistic interactions and switching
  topologies.
\newblock {\em Automatica}, 73:110--116, 2016.

\bibitem{He14}
J.~M. Hendrickx.
\newblock A lifting approach to models of opinion dynamics with antagonisms.
\newblock In {\em Proc. of the 53th IEEE Conference on Decision and Control},
  pages 2118--2123, 2014.

\bibitem{LiChBaBe15}
J.~Liu, X.~Chen, T.~Ba\c{s}ar, and M.-A. Belabbas.
\newblock Stability of discrete-time {A}ltafini's model: {A} graphical
  approach.
\newblock In {\em Proc. of the 54th IEEE Conference on Decision and Control},
  pages 2835--2840, 2015.

\bibitem{XiCaJo16}
W.~Xia, M.~Cao, and K.~H. Johansson.
\newblock Structural balance and opinion separation in trust-mistrust social
  networks.
\newblock {\em IEEE Transcations on Control of Network Systems}, 3(1):46--56,
  2016.

\bibitem{Fr11}
N.~E. Friedkin.
\newblock A formal theory of reflected appraisals in the evolution of power.
\newblock {\em Administrative Science Quarterly}, 56(4):501--529, 2011.

\bibitem{XiLiJoBa16}
W.~Xia, J.~Liu, K.~H. Johansson, and T.~Ba\c{s}ar.
\newblock Convergence rate of the modified degroot-friedkin model with doubly
  stochastic relative interaction matrices.
\newblock In {\em Proc. of the 2016 American Control Conference}, pages
  1054--1059, 2016.

\bibitem{Lo07}
J.~Lorenz.
\newblock Continuous opinion dynamics under bounded confidence: {A} survey.
\newblock {\em International Journal of Modern Physics C}, 18:1819--1838, 2007.

\bibitem{BlHeTs10}
V.~D. Blondel, J.~M. Hendrickx, and J.~N. Tsitsiklis.
\newblock Continuous-time average-preserving opinion dynamics with
  opinion-dependent communications.
\newblock {\em SIAM Journal on Control and Optimization}, 48:5214--5240, 2010.

\bibitem{ChLiBeXuBa17}
X.~Chen, J.~Liu, M.-A. Belabbas, Z.~Xu, and T.~Ba\c{s}ar.
\newblock Distributed evaluation and convergence of self-appraisals in social
  networks.
\newblock {\em IEEE Transactions on Automatic Control}, 62(1):291--304, 2017.

\bibitem{ReBe05}
W.~Ren and R.~W. Beard.
\newblock Consensus seeking in multiagent systems under dynamically changing
  interaction topologies.
\newblock {\em IEEE Transactions on Automatic Control}, 50(5):655--661, 2005.

\bibitem{GoRo01}
C.~Godsil and G.~Royle.
\newblock {\em Algebraic Graph Theory}.
\newblock Springer-Verlag, New York, 2001.

\bibitem{OlMu04}
R.~Olfati-Saber and R.~M. Murray.
\newblock Consensus problems in networks of agents with switching topology and
  time-delays.
\newblock {\em IEEE Transactions on Automatic Control}, 49(9):1520--1533, 2004.

\bibitem{XiWa08}
F.~Xiao and L.~Wang.
\newblock Asynchronous consensus in continuous-time multi-agent systems with
  switching topology and time-varying delays.
\newblock {\em IEEE Transactions on Automatic Control}, 53:1804--1816, 2008.

\bibitem{HeTs13}
J.~M. Hendrickx and J.~Tsitsiklis.
\newblock Convergence of type-symmetric and cut-balanced consensus seeking
  systems.
\newblock {\em IEEE Transactions on Automatic Control}, 58(1):214--218, 2013.

\bibitem{MaGi13}
S.~Martin and A.~Girard.
\newblock Continuous-time consensus under persistent connectivity and slow
  divergence of reciprocal interaction weights.
\newblock {\em SIAM Journal on Control and Optimization}, 51(3):2568--2584,
  2013.

\bibitem{ShJo13d}
G.~Shi and K.~H. Johansson.
\newblock Robust consensus for continuous-time multiagent dynamics.
\newblock {\em SIAM Journal on Control and Optimization}, 51(5):3673--3691,
  2013.

\bibitem{Da66}
J.~Danskin.
\newblock The theory of max-min, with applications.
\newblock {\em SIAM Journal on Applied Mathematics}, 14:641--664, 1966.

\bibitem{LiFrMa07}
Z.~Lin, B.~Francis, and M.~Maggiore.
\newblock State agreement for continuous-time coupled nonlinear systems.
\newblock {\em SIAM Journal on Control and Optimization}, 46(1):288--307, 2007.

\end{thebibliography}

\end{document}